\documentclass[envcountsame,runningheads]{llncs}
\title{$2$-word-$\pi$-representable Graphs} 

\usepackage{todonotes}
\usepackage{hyperref} 
\usepackage{amsmath,amsfonts,amssymb,mathtools} 
\usepackage{cleveref}
\Crefname{algocf}{Algorithm}{Algorithms}

\usepackage[T1]{fontenc}

\usepackage{graphicx}

\usepackage[linesnumbered,vlined,ruled]{algorithm2e}

\newif\ifpaper
\paperfalse 

\usepackage{complexity}
\usepackage{xspace}

\crefalias{proposition}{proposition}

\DeclareSymbolFont{Shuffle}{U}{shuffle}{m}{n}
\DeclareFontFamily{U}{shuffle}{}
\DeclareFontShape{U}{shuffle}{m}{n}{%
  <-8>shuffle7%
  <8->shuffle10%
}{}
\DeclareMathSymbol\shuffle{\mathbin}{Shuffle}{"001}
\DeclareMathSymbol\cshuffle{\mathbin}{Shuffle}{"002}

\def\nth#1{#1$^{\mbox{\tiny th}}$}

\def\ta{\mathtt{a}}
\def\tb{\mathtt{b}}
\def\tc{\mathtt{c}}
\def\td{\mathtt{d}}
\def\te{\mathtt{e}}
\def\tf{\mathtt{f}}
\def\tg{\mathtt{g}}

\def\tr{\mathtt{r}}
\def\ts{\mathtt{s}}
\def\tx{\mathtt{x}}

\def\proj#1#2{\pi_{#1,#2}}

\DeclareMathOperator{\alphabet}{alph}
\DeclareMathOperator{\al}{\alphabet}
\DeclareMathOperator{\letters}{\alphabet}
\DeclareMathOperator{\Fact}{Fact}

\DeclareMathOperator{\IE}{IE}
\DeclareMathOperator{\IES}{IES}

\DeclareMathOperator{\Deg}{deg}

\newcommand{\N}{\mathbb{N}}

\newcommand{\emptyWord}{\varepsilon}

\renewcommand{\alph}{\alphabet}

\newcommand{\longversion}[1]{#1}
\newcommand{\shortversion}[1]{}

\longversion{\usepackage[appendix=inline]{apxproof}}
\shortversion{\usepackage{apxproof}}%
\shortversion{}

\ifpaper
\newcommand{\colorOne}{black}
\newcommand{\colorThree}{black}
\newcommand{\colorTwo}{black}
\newcommand{\colorFour}{black}
\newcommand{\colorFive}{black}
\newcommand{\colorSix}{black}
\else
\definecolor{darkRed}{RGB}{182, 31, 31}
\newcommand{\colorOne}{darkRed}
\definecolor{darkOrange}{RGB}{218, 78, 1}
\newcommand{\colorThree}{darkOrange}
\definecolor{darkYellow}{RGB}{239, 108, 0}
\newcommand{\colorTwo}{darkYellow}
\definecolor{darkGreen}{RGB}{0, 102, 0}
\newcommand{\colorFour}{darkGreen}
\newcommand{\colorFive}{blue}
\definecolor{darkPurple}{RGB}{106, 27, 154}
\newcommand{\colorSix}{darkPurple}
\fi

\usepackage{currfile}
\usepackage{lineno}
\usepackage{filehook}
\usepackage{xstring}

\newtheorem{construction}[theorem]{Construction}

\newcommand{\filestack}{{main.tex}, bottom}  

\newif\ifinwork
\inworkfalse

\ifinwork
\makeatletter
\def\getcallerfile#1,#2\relax{%
  \def\callerfile{#2}%
  \texttt{ ← #1:\the\inputlineno}%
}
\def\popfilefromstack#1,#2\relax{%
  \xdef\filestack{#2}%
}
\makeatother

\AtBeginOfInputs{%
  \noindent\textcolor{red}{%
    \texttt{Begin file: \currfilename}%
    \expandafter\getcallerfile\filestack\relax
  }\\
  \xdef\filestack{{\currfilename},\filestack}%
}

\AtEndOfInputs{%
  \expandafter\popfilefromstack\filestack\relax
  \noindent\textcolor{red}{\texttt{End file: \currfilename}}\\[0.5em]%
}
\fi

\definecolor{lime}{HTML}{A6CE39}
\DeclareRobustCommand{\orcidicon}{%
	\begin{tikzpicture}
		\draw[lime, fill=lime] (0,0) 
		circle [radius=0.16] 
		node[white] {{\fontfamily{qag}\selectfont \tiny ID}};
		\draw[white, fill=white] (-0.0625,0.095) 
		circle [radius=0.007];
	\end{tikzpicture}
	\hspace{-2mm}
}

\foreach \x in {A, ..., Z}{%
	\expandafter\xdef\csname orcid\x\endcsname{\noexpand\href{https://orcid.org/\csname orcidauthor\x\endcsname}{\noexpand\orcidicon}}
}

\definecolor{DarkGreen}{RGB}{0,100,0}

\title{$2$-word-$\pi$-representable Graphs}
\author{Duncan Adamson\inst{1}\orcidD{}\and Amanita Dietz\inst{2} \and
	Pamela Fleischmann\inst{2}\orcidA{} \and
	Annika Huch\inst{2}\orcidB{} \and
	Silas Cato Sacher\inst{3}\orcidC{}}
\authorrunning{D. Adamson et al.}
\institute{University of St Andrews, UK\\
	\email{duncan.adamson@st-andrews.ac.uk}\and Department of Computer Science, Kiel University, Germany
	\email{stu215770@mail.uni-kiel.de, \{fpa,ahu\}@informatik.uni-kiel.de} \and
	Universit\"at Trier, Fachbereich IV, Informatikwissenschaften, Germany\\
	\email{sacher@uni-trier.de}}

\begin{document}
	
	\maketitle              
	%
	\begin{abstract}
This paper investigates the new notion of $2$-word-$\pi$-repre\-sentable graphs: the nodes of the graph correspond to the letters of the two words and there exists an edge between two nodes if the projections of any two letters of both words are equal. The benefit of not only using one word for a representation as introduced by Kitaev and Pyatkin is that every graph is $2$-word-$\pi$-representable. We present an algorithm that returns two representing words for any graph.
Aside, we show that every permutation graph is representable by two $1$-uniform words and give constructions how graph operations on $2$-word-$\pi$-representable graphs can be realised on their representing words which give further insights into the representation of cographs. 
\end{abstract}

	\paragraph{Acknowledgement.} We would like to thank Philipp Kindermann for patiently sharing his deep knowledge about graph classes with us.
	\section{Introduction}
	In the age of big data, an increasingly common problem is that of storage space, the challenge of representing large data sets in as concise a manner as possible. Graphs in particular form one of the most commonly used and fundamental structures, and one that suffers from the problem of storage acutely, with the most common representation, the adjacency matrix, requiring space quadratic relative to the number of nodes within the graph. One commonly studied solution introduced by Kitaev and Pyatkin \cite{KitPya2008} to this problem are \emph{word-representable graphs}, a means of representing graphs through an alphabet of nodes. Initially motivated by computational algebra, these have been studied across a wide spectrum of areas in computer science,  including graph theory, language theory \cite{FenFFKS2026,FenFFMS2026}, and even in combinatorics on words \cite{das2024squarefreewordrepresentationwordrepresentablegraphs,FleHLN2024,böll2025wordrepresentablegraphslocalitywords,das2025wordrepresentabilitycobipartitegraph,das2025pcompletesquarefreewordrepresentationwordrepresentable}. 


A graph is \emph{word-representable} if there exists a word whose letters are the nodes and the edges in the graph exactly correspond to the pairs of letters that are alternating in the word: two distinct letters are called \emph{alternating} if the removal of all other letters does not yield any two consecutive occurrences of the same letter. See \cite{KitLoz2015,Kit2017} for an in-depth overview. For instance, the word $\mathtt{rescues}$ represents the graph from \Cref{fig:graphwordrepKitaev}: looking at the nodes (letters) $\te$ and $\ts$, we get
the projection $\mathtt{eses}$ and since $\te$ and $\ts$ alternate, we have an edge; on the other hand if we
take $\te$ and $\tr$, we get the projection $\mathtt{ree}$ and since we have the square $\te\te$, we do not have an
edge between $\te$ and $\tr$.

\begin{figure}
	\begin{minipage}{.4\textwidth}
		\centering
		\begin{tikzpicture}[scale=0.55]
			\node (s) at (-2-1,0+1) {$G$};
			\node[shape=circle,draw=black] (s) at (-2,0) {$\mathtt s$};
			\node[shape=circle,draw=black] (e) at (-2,-2) {$\mathtt e$};
			\node[shape=circle,draw=black] (c) at (0,0) {$\mathtt c$};
			\node[shape=circle,draw=black] (r) at (2,0) {$\mathtt r$};
			\node[shape=circle,draw=black] (u) at (2,-2) {$\mathtt u$};
			
			\path [-] (s) edge node[left] {} (e);
			\path [-] (s) edge node[left] {} (u);
			\path [-] (s) edge node[left] {} (c);
			\path [-] (e) edge node[left] {} (c);
			\path [-] (c) edge node[left] {} (r);
			\path [-] (c) edge node[left] {} (u);
			\path [-] (r) edge node[left] {} (u);
			\path [-] (e) edge node[left] {} (u);
		\end{tikzpicture} \\
		$w = \mathtt{rescues}$
		\caption{$G$ represented by $w$.}
		\label{fig:graphwordrepKitaev}
	\end{minipage}%
	\begin{minipage}{.6\textwidth}
		\centering
		\begin{tikzpicture}[scale=0.55]
			\node (s') at (-2+8-1,0+1) {$G'$};
			\node[shape=circle,draw=black] (s') at (-2+8,0) {$\mathtt s$};
			\node[shape=circle,draw=black] (e') at (-2+8,-2) {$\mathtt e$};
			\node[shape=circle,draw=black] (c') at (0+8,0) {$\mathtt c$};
			\node[shape=circle,draw=black] (r') at (2+8,0) {$\mathtt r$};
			\node[shape=circle,draw=black] (u') at (2+8,-2) {$\mathtt u$};
			
			\path [-] (e') edge node[left] {} (c');
			\path [-] (e') edge node[left] {} (u');
			\path [-] (s') edge node[left] {} (u');
			\path [-] (s') edge node[left] {} (c');
			\path [-] (c') edge node[left] {} (u');
		\end{tikzpicture} \\
		$w = \mathtt{rescues}$, $v = \mathtt{secures}$
		\caption{$G'$ $2$-word-$\pi$-represented by $w$ and $v$.}
		\label{fig:2-word-pi-represented}
	\end{minipage}

\end{figure}
This definition allows many graph classes to be represented in essentially linear space relative to the number of nodes, a stark improvement on the quadratic size required by the adjacency matrix.
Besides the various applications of word-representablility in periodic scheduling \cite{HalKitPya2010,KitLoz2015}, topology \cite{DBLP:journals/dcg/OliverosT25}, and the power domination problem in physics \cite{chandrasekaran2019k}, word representable graphs are also of a wide theoretical interest. They generalise several graph classes such as circle graphs, $3$-colourable graphs, and comparability graphs and possess interesting properties like the polynomial time solvability of their maximum clique problem \cite{HalKitPya2016,KitLoz2015}. Moreover, there is fundamental research about grid graphs
(cf. e.g., \cite{glen2018colourability}) and split graphs (cf. e.g., \cite{kitaev2021semitransitive}).
Staying in the field of word-representability by itself, it was shown that a graph is word-representable iff it admits a semi-transitive orientation \cite{HalKitPya2016}.
In \cite{FleHLN2024}, the authors connected the set of $k$-local words and word-representable graphs. This line of research was continued in \cite{böll2025wordrepresentablegraphslocalitywords}.

Nevertheless, not every graph is word-representable in the classical sense. The smallest non-word representable graph is the wheel graph $W_5$ on $6$ nodes~\cite{KitPya2008,KitLoz2015} (see \Cref{graph:w5}). 
An asymptotic result about the number of word-representable graphs can be found in \cite{DBLP:journals/dam/CollinsKL17}.
Importantly, the decision problem whether a graph is word-representable by this definition is NP-complete.
As a result of this challenge, other approaches to represent graphs have become increasingly studied. At a high level,  alternation can be seen as an avoidance of squares, and thus several works \cite{DBLP:journals/combinatorics/JonesKPR15,kitaev2015existence} investigated graphs represented by words avoiding other patterns. Alternatively, one can relax the alternation condition, allowing some number of violations of alternations in the representing word \cite{cheon2018k11representable,On11AndMulti}. One may also move beyond representing graphs by a single word, to instead do so by a set at most $k$ words, each representing a word-representable subgraph of $G$, such that their union is $G$ \cite{Multi,Multi2,On11AndMulti}. Finally, recent work has looked at word-representable graphs from a language theoretic point of view is presented in \cite{FenFFKS2026,FenFFMS2026}.

\begin{figure}
	\centering
	
	\begin{minipage}{.28\textwidth}
		\centering
		\begin{tikzpicture}
			\node (a) {$\ta$};
			\node[above right=0.05cm and 0.05cm of a] (c) {$\tc$};
			\node[above left=0.05cm and 0.05cm of a] (b) {$\tb$};
			\node[below=0.05cm of a] (d) {$\td$};
			\node[below left= 0.3cm and 1cm of a] (e) {$\te$};
			\node[above =0.9cm of a] (f) {$\tf$};
			\node[below right= 0.3cm and 1cm of a] (g) {$\tg$};
			\draw (a) -- (b);
			\draw (a) -- (c);
			\draw (a) -- (d);
			\draw (a) -- (e);
			\draw (a) -- (f);
			\draw (a) -- (g);
			\draw (b) -- (c);
			\draw (b) -- (d);
			\draw (b) -- (e);
			\draw (c) -- (d);
			\draw (c) -- (f);
			\draw (d) -- (g);
			\draw (e) -- (f);
			\draw (e) -- (g);
			\draw (f) -- (g);
		\end{tikzpicture}
		\caption{}
		\label{graph:non-representable_2}
		$w=\mathtt{bfcgdedfa}$\\
		$v=\mathtt{fgbecdfda}$
	\end{minipage}%
	\begin{minipage}{.28\textwidth}
		\centering
		\begin{tikzpicture}
			\node (a) {$\ta$};
			\node[above right=0.05cm and 0.05cm of a] (c) {$\tc$};
			\node[above left=0.05cm and 0.05cm of a] (b) {$\tb$};
			\node[below=0.05cm of a] (d) {$\td$};
			\node[below left= 0.3cm and 1cm of a] (e) {$\te$};
			\node[above =0.9cm of a] (f) {$\tf$};
			\node[below right= 0.3cm and 1cm of a] (g) {$\tg$};
			\draw (a) -- (b);
			\draw (a) -- (c);
			\draw (a) -- (d);
			\draw (a) -- (e);
			\draw (a) -- (f);
			\draw (a) -- (g);
			\draw (b) -- (e);
			\draw (c) -- (f);
			\draw (d) -- (g);
			\draw (e) -- (f);
			\draw (e) -- (g);
			\draw (f) -- (g);
		\end{tikzpicture}
		\caption{}
		\label{graph:non-representable_3}
		$w=\mathtt{edbgdcfa}$\\
		$v=\mathtt{dcegfdba}$
	\end{minipage}%
	\begin{minipage}{.28\textwidth}
		\centering
		\begin{tikzpicture}
			\node (a) {$\ta$};
			\node[above right=0.05cm and 0.05cm of a] (c) {$\tc$};
			\node[above left=0.05cm and 0.05cm of a] (b) {$\tb$};
			\node[below=0.05cm of a] (d) {$\td$};
			\node[below left= 0.3cm and 1cm of a] (e) {$\te$};
			\node[above =0.9cm of a] (f) {$\tf$};
			\node[below right= 0.3cm and 1cm of a] (g) {$\tg$};
			\draw (a) -- (b);
			\draw (a) -- (c);
			\draw (a) -- (d);
			\draw (a) -- (e);
			\draw (a) -- (f);
			\draw (a) -- (g);
			\draw (b) -- (e);
			\draw (c) -- (d);
			\draw (c) -- (f);
			\draw (d) -- (g);
			\draw (e) -- (f);
			\draw (e) -- (g);
			\draw (f) -- (g);
		\end{tikzpicture}
		\caption{}
		\label{graph:non-representable_4}
		$w=\mathtt{febdgcfa}$\\
		$v=\mathtt{dfcegfba}$
	\end{minipage}
\end{figure}

\paragraph{Own Contribution.} 
To obtain a more powerful definition of word-representability, we generalised this notion on two words (distinct to the work done on words as subgraphs \cite{Multi,Multi2}). A graph is \emph{$2$-word-$\pi$-representable} if there exist two words $w,v$ having all the graph's nodes as letters such that the projection
$\pi_{\ta,\tb}$ on any two letters $\ta,\tb \in \Sigma$ of $w$ and $v$ are equal iff there is an edge between the corresponding nodes. Notice that the association of vertices and the representing word's alphabet is kept analogous as in the classical word-representability. Distinct from the classic notion, we are generalising the representablility to two words: instead of the fixed pattern of alternation we consider the projections of two words. Equal projections are in our opinion the easiest way to describe edges using two words.
For example, the graph represented by $\mathtt{rescues}$ and $\mathtt{secures}$ is given in \Cref{fig:2-word-pi-represented}.
Looking at $\ts$ and $\te$, we get the projections $\pi_{\te,\ts}(\mathtt{rescues})=\mathtt{eses}$ and
$\pi_{\te,\ts}(\mathtt{secures})=\mathtt{sees}$. Since they differ, there is no edge between $\te$ and $\ts$. On the
other hand, we have $\pi_{\te,\tc}(\mathtt{rescues})=\te\tc\te$ and $\pi_{\te,\tc}(\mathtt{secures})=\te\tc\te$ and therefore, we have an edge between $\te$ and $\tc$.
Note that the four graphs in \Cref{graph:non-representable_2,graph:non-representable_3,graph:non-representable_4,graph:w5} are non-representable by the classic model~\cite{KitPya2008}, showing the power of out approach.
Beyond the introduction of $2$-word-$\pi$-representable words, we prove that every graph $G= (V,E)$ is $2$-word-$\pi$-representable by two words of length $|V| \cdot |\overline{E}|$ where $\overline{E}$ denotes the complement of the edges of $G$ w.r.t. the complete graph. Further, we present an algorithm for constructing a pair of representing words. Since the na\"ive construction might result in words of even worse length as the adjacency matrix' size, we follow the line of research of classical word-representability \cite{KitPya2008} and investigate $k$-uniform representabiliy by two $k$-uniform words (every letter occurs exactly $k$ times in each word). We show that the class of graphs that are representable by two $1$-uniform words is exactly the group of permutation graphs. 
Moreover, we present how basic graph operations like union and join can be realised for $2$-word-$\pi$-representable graphs by altering the respective words representing the single graphs. Lastly, we give insights into $2$-uniformly $2$-word-$\pi$-representable graphs.

\paragraph{Structure of the work.} 
In \Cref{prelims} we give all the necessary definitions needed for our research. In \Cref{construction1} we show that
every graph is representable by two words with our definitions and we give algorithms to construct these words. Moreover,
we show that the $1$-uniform $2$-word-$\pi$-representable graphs are exactly the permutation graphs. In \Cref{operations}
we investigate operations on graphs.

	\section{Preliminaries}\label{prelims}
	\setcounter{theorem}{0}

Let $\N$ be the natural numbers starting with $1$ and let $\N_0=\N\cup\{0\}$. Define $[n]=\{1,\ldots,n\}$ for all $n\in\N$ and $[n]_0=[n]\cup\{0\}$. For a set $M$ and $k\in\N$
define $\binom{M}{k}=\{S\subseteq M\mid k = |S|\}$. A {\em permutation} is bijection from a set to itself. For $n \in \N$ and a permutation $\sigma: [n] \to [n]$ we use the {\em one-line notation} $\sigma = (\sigma(1)\sigma(2)\dots \sigma(n))$. For example, we denote $1 \mapsto 3, 2 \mapsto 2, 3 \mapsto 1$ by $(321)$. 


\noindent
{\bf Language Theory.}
An {\em alphabet} is a finite set $\Sigma$. The elements of $\Sigma$ are called {\em letters} and the finite concatenation
of letters is called a {\em word}. Let $\Sigma^{\ast}$ be the set of all finite words over $\Sigma$ including the 
{\em empty word} $\emptyWord$ which does not contain any letter. The \emph{alphabet of the word} $w$, denoted 
by $\alph(w)$, is the set of letters contained in the word $w$. The \emph{length} of $w$ is the number of 
letters of $w$ and denoted by $|w|$. Let $\Sigma^k$ be the set of all words over the alphabet $\Sigma$ with 
length $k\in\N$.
Denote the $i^\text{th}$ letter of $w\in\Sigma^{\ast}$ by $w[i]$ for all $i\in[|w|]$ and define $|w|_{\ta}=|\{i\in[|w|]|\,w[i]=\ta\}$ for all $w\in\Sigma^{\ast}$ and $\ta\in\Sigma$.
For a given $k \in \N$, $w \in \Sigma^*$ is \emph{$k$-uniform} if $|w|_\ta = k$ for all $\ta \in \Sigma$.
For all $i,j\in\N$ with $i\leq j\leq|w|$, the word $w[i]\cdots w[j]$ is called a {\em factor} of $w\in\Sigma^{\ast}$ and
denoted by $w[i\ldots j]$. $\Fact(w)$ denotes the set of all factors of a given word $w\in\Sigma^{\ast}$.
Define $w^R=w[|w|]\cdots w[1]$ as the \emph{reverse} of the word $w$. 
A mapping $f:\Sigma^{\ast}\rightarrow\Sigma^{\ast}$
is called {\em morphic} if $f(xy)=f(x)f(y)$ holds for all $x,y\in\Sigma^{\ast}$. Notice that a morphism is well-defined if $f(\ta)$ is given for all $\ta\in\Sigma$. A special morphism is the {\em projection} onto two letters $\ta,\tb\in\Sigma$ defined by $\proj{\ta}{\tb}(\ta)=\ta$, $\proj{\ta}{\tb}(\tb)=\tb$, and $\proj{\ta}{\tb}(x)=\emptyWord$ for $x\in\Sigma\backslash\{\ta,\tb\}$.

\noindent
{\bf Graph Theory.}
A (undirected) \emph{graph}, $G=(V,E)$, consists of a set of \emph{vertices} $V$ and a set of \emph{edges} $E$ with
$E\subseteq \binom{V}{2}$. 
Two graphs $G = (V,E)$ and $H = (V',E')$ are \emph{isomorphic} if there exists a bijection $f: V \to V'$ such that $\{u,v\} \in E$ iff $\{f(u), f(v)\} \in E'$. For a given graph $G$ let $V(G)$ resp. $E(G)$ denote its sets of vertices and edges. In the case $E=\binom{V}{2}$, $G$ is called {\em complete}. $G$ is called {\em empty} if $E=\emptyset$. For all $v\in V$ define the {\em neighbours} of $V$ w.r.t. $G$ by $N_G(v)=\{u\in V|\,\{u,v\}\in E\}$. A vertex $v\in V$ is called {\em isolated}
if $N_G(v)=\emptyset$, and {\em universal} if $N_G(v)\cup\{v\}=V$. Let $\Deg(v)=|N_G(v)|$ denote the {\em degree} of $v\in V$ and define the maximal degree of $G$ by $\Delta(G)=\max_{v\in V}\Deg(v)$. Two edges $e_1,e_2\in E$ are called {\em independent} iff $e_1\cap e_2=\emptyset$.
Given a graph $G=(V,E)$ define the {\em complement graph} by $\overline{G}=(V,\binom{V}{2}\backslash E)$. For two given graphs $G=(V,E), G'=(V',E')$ define their {\em union} $G \cup G'$ by $V(G \cup G') = V \cup V'$ and $E(G \cup C')=E\cup E'$ and define their {\em join} $G \nabla G'$ by $V(G\nabla G') = V \cup V'$ and $E(G\nabla G')=E\cup E'\cup \{\{u,u'\} \mid u \in V, u' \in V')\}$.
Let $C_n$ for $n \in \N$ denote the {\em cycle graph} of length $n$. A graph $G = (V,E)$ is a {\em subgraph} of $G' = (V',E')$ iff $V \subseteq V'$ and $E \subseteq E'$. $G$ is an {\em induced subgraph} of $G'$ iff $E = \binom{V}{2} \cap E'$ and $G$ is a {\em spanning subgraph} of $G'$ iff $V = V'$. Further, a class of graphs $\mathcal G$ is called \emph{hereditary} if it is closed under induced subgraphs. The \emph{speed} if a class of graphs $\mathcal G$ is the function $n \mapsto |\mathcal{G}_n|$ where $\mathcal{G}_n$ describes the set of graphs in $\mathcal{G}$ with vertex set $[n]$.



%
%


\begin{definition}[2-Word-$\pi$-Representable Graphs]\label{def:graph_p}
	Given a pair of words $w,v\in\Sigma^\ast$ with $\alph(w)=\alph(v)=A$, define the graph $G(w,v)=(A,E)$ by $\{\ta,\tb\}\in E$ if and only if $\proj{\ta}{\tb}(v)=\proj{\ta}{\tb}(w)$ for all $\ta,\tb\in A$. If there exists $w,v\in\Sigma^{\ast}$ such that for a given graph $G$, we have $G=G(w,v)$ then $G$ is called {\em 2-word-$\pi$-representable}.
\end{definition}

Note that this representation is symmetric, i.e., $G(w,v) = G(v,w)$ in the above definition.
The smallest graph that is not word-representable in the classical sense (cf. \cite{KitPya2008,KitLoz2015})) is the wheel graph $W_5$
depicted in \Cref{graph:w5}. It can be represented by $w=\mathtt{acbdecef}$ and $v=\mathtt{cdaebecf}$.

    \begin{figure}
        \centering
        \begin{tikzpicture}
            \node (a) {$\ta$};
            \node[below right= 0.6cm and 0.1cm of a] (e) {$\te$};
            \node[right of=a] (f) {$\tf$};
            \node[above of=f] (b) {$\tb$};
            \node[right of=f] (c) {$\tc$};
            
            \node[below left=0.6cm and 0.1cm of c] (d) {$\td$};

            \draw (a) -- (b);
            \draw (a) -- (e);
            \draw (a) -- (f);
            \draw (b) -- (c);
            \draw (b) -- (f);
            \draw (c) -- (d);
            \draw (c) -- (f);
            \draw (d) -- (e);
            \draw (d) -- (f);
            \draw (e) -- (f);
        \end{tikzpicture}
        \caption{$W_5$ represented by  $w=\mathtt{acbdecef}$ and $v=\mathtt{cdaebecf}$.}
        \label{graph:w5}
      \end{figure}%

%


	\section{Construction of Words for a given Graph}\label{construction1}
	In this section, we show that every graph is $2$-word-$\pi$-representable. We start with a couple of auxiliary lemmata. The first result shows that there are infinitely many pairs of words representing the same graph which motivates our search for the shortest pair of two words that we pursue later within this section.

\begin{remark}\label{thm:append_word}
	For all $w,v,u\in\Sigma^\ast$ with $\alph(w)=\alph(v)\supseteq\alph(u)$ we have $G(w,v)=G(wu,vu) = G(uw,uv)$.
	First, notice that  both graphs have the same vertices.
	Since $\proj{\ta}{\tb}(wu)=\proj{\ta}{\tb}(w)\proj{\ta}{\tb}(u)$ and $\proj{\ta}{\tb}(vu)=\proj{\ta}{\tb}(v)\proj{\ta}{\tb}(u)$, we have 
	$\proj{\ta}{\tb}(wu)=\proj{\ta}{\tb}(vu)$ iff $\proj{\ta}{\tb}(w)=\proj{\ta}{\tb}(v)$. 
	Analogously, we obtain $G(w,v) = G(uw,uv)$, concluding the proof. Note that $\alph(w)\supseteq\alph(u)$
	is only necessary to preserve the vertex set.
\end{remark}


Note that \Cref{thm:append_word} holds regardless of whether $u$ is inserted at the front or at the end of pair of words, that is, $G(w, v) = G(u w, u v) = G(w u, v u)$.
The following lemma shows that by appending a $1$-uniform word to $w$ and $v$ resp., we can delete an edge from $G(w,v)$ for some given $w,v\in\Sigma^{\ast}$. To achieve this the two letters corresponding to the deleted edge must have a different order in the words appended to $w$ and $v$. The following
result is weaker than a later result, but it allows an intuitive insight into the deletion of edges.

\ifpaper
\else
\begin{lemma}\label{thm:remove_edge}
	Let $w,v\in\Sigma^\ast$ with $\letters(w)=\letters(v)$ and $\{\ta,\tb\}\in E(G(w,v))$ for some $\ta,\tb\in\letters(w)$. Then, there exists a pair of  $1$-uniform words $u_w,u_v\in\Sigma^\ast$, such that $E(G(wu_w,vu_v))=E(G(w,v))\setminus\{\{\ta,\tb\}\}$.
\end{lemma}

\begin{proof}
Assume first $|\letters(w)|>2$.
	Let $\ta,\tb\in\alph(w)$ and 
	$x\in(\alph(w)\setminus\{\ta,\tb\})^\ast$ with $|x|_\tc=1$ for all $\tc\in\alph(w)\setminus\{\ta,\tb\}$. 
	Let $u_w=x\ta\tb\quad\mbox{and}\quad u_v=x\tb\ta$.
	These words are $1$-uniform by construction. Let $\tc\in\letters(w)\backslash\{\ta,\tb\}$ and $\td\in\letters(w)\backslash\{\tc\}$. Then we obtain 
	\[
	\proj{\tc}{\td}(wu_w)=\proj{\tc}{\td}(vu_v)\quad\mbox{iff}\quad\proj{\tc}{\td}(w)=\proj{\tc}{\td}(v)\mbox{ and }\proj{\tc}{\td}(u_w)=\proj{\tc}{\td}(u_v). 
	\]
	Since the latter one holds by construction, we have $$\{\tc,\td\}\in E(G(vu_v,wu_w))\quad\mbox{iff}\quad \{\tc,\td\}\in E(G(v,w))$$.
	Also by construction, we get that $\proj{\ta}{\tb}(wu_w)\neq \proj{\ta}{\tb}(v u_v)$ which concludes the proof for $|\letters(w)|>2$.
	
	Assume now $\letters(w)=\{\ta,\tb\}$. In this case $u_w=\ta\tb$ and $u_v=\tb\ta$ witness the claim.\qed
\end{proof}

\fi

With \Cref{thm:remove_edge} we prove that every graph is $2$-word-$\pi$-representable: start with a complete graph containing the required vertices and successively delete all edges that are not needed.

\ifpaper
\else
\begin{theorem}\label{thm:graph_construction}
	For every graph $G=(V,E)$ there exist $w,v\in V^\ast$ with $\alph(w)=\alph(v)=V$, such that $G(w,v)=G$.
\end{theorem}

\begin{proof}\label{proof:graph_construction}
Clearly, the complete graph $K_n$ on the vertices $V(K_n) = \{ v_1, \dots, v_n \}$ for some $n \in \N$ is $2$-word-$\pi$-represented by $w := v_1 \cdots v_n$ and $v := w$. 
For every other graph $G = (V,E)$, the words $w,v$ are constructed iteratively by removing one edge after another with \Cref{thm:remove_edge}, starting with $w=v=\emptyWord$. After removing all edges $\{\ta,\tb\}\in\overline{E}$ from the edges $E(K_n)$ of the complete graph $K_n$, only the edges in $E$ are left. This takes $|\overline{E}|$ iterations and since in each iteration $|V|$ letters are added to the words $w,v$, they are of length $|w|=|v|=|\overline{E}|\cdot|V|$. \qed
\end{proof}
\fi

\begin{remark}
Since all graphs are $2$-word-$\pi$-representable the research on {\em nice} representations is motivated. Following the classical word-representability research, we focus on the representation by uniform words.
\end{remark}


\Cref{fig:example_graph_construction} shows an exemplary construction for $G=(V,E)$ with vertices $V=\{\ta,\tb,\tc,\td\}$ and edges $E=\{\{\ta,\tb\},\{\ta,\td\},\{\tc,\td\}\}$. Every step shows the corresponding graph for each iteration (where the new factor for every iteration is separated by a $\cdot$ from the others for a better visibility).

\begin{figure}
	\begin{minipage}{.25\textwidth}
		\centering
		\begin{tikzpicture}
			\node (a) {$\ta$};
			\node[right of=a] (b) {$\tb$};
			\node[below of=a] (c) {$\tc$};
			\node[below of=b] (d) {$\td$};
			\draw (a) -- (b);
			\draw (a) -- (d);
			\draw (a) -- (c);
			\draw (b) -- (c);
			\draw (b) -- (d);
			\draw (c) -- (d);
		\end{tikzpicture}
		
		$w=\emptyWord$\\
		$v=\emptyWord$
	\end{minipage}%
	\begin{minipage}{.25\textwidth}
		\centering
		\begin{tikzpicture}
			\node (a) {$\ta$};
			\node[right of=a] (b) {$\tb$};
			\node[below of=a] (c) {$\tc$};
			\node[below of=b] (d) {$\td$};
			\draw (a) -- (b);
			\draw (a) -- (d);
			\draw (b) -- (c);
			\draw (b) -- (d);
			\draw (c) -- (d);
		\end{tikzpicture}
		
		$w=\tb\td\ta\tc$\\
		$v=\tb\td\tc\ta$
	\end{minipage}%
	\begin{minipage}{.25\textwidth}
		\centering
		\begin{tikzpicture}
			\node (a) {$\ta$};
			\node[right of=a] (b) {$\tb$};
			\node[below of=a] (c) {$\tc$};
			\node[below of=b] (d) {$\td$};
			\draw (a) -- (b);
			\draw (a) -- (d);
			\draw (b) -- (d);
			\draw (c) -- (d);
		\end{tikzpicture}
		$w=\tb\td\ta\tc\cdot\ta\td\tb\tc$\\
		$v=\tb\td\tc\ta\cdot\ta\td\tc\tb$
		
	\end{minipage}%
	\begin{minipage}{.25\textwidth}
		\centering
		\begin{tikzpicture}
			\node (a) {$\ta$};
			\node[right of=a] (b) {$\tb$};
			\node[below of=a] (c) {$\tc$};
			\node[below of=b] (d) {$\td$};
			\draw (a) -- (b);
			\draw (a) -- (d);
			\draw (c) -- (d);
		\end{tikzpicture}
		
		$w=\tb\td\ta\tc\ta\td\tb\tc\cdot\ta\tc\tb\td$\\
		$v=\tb\td\tc\ta\ta\td\tc\tb\cdot\ta\tc\td\tb$
	\end{minipage}
	\caption{Construction of $w,v$ for $G$.}
	\label{fig:example_graph_construction}
\end{figure}
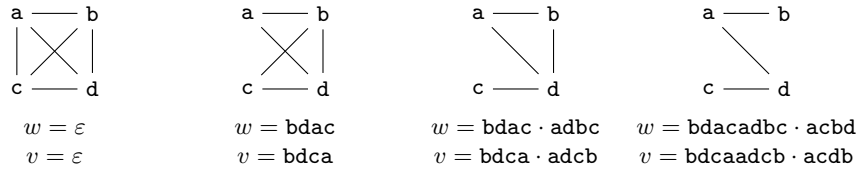


The intuition behind the construction in \Cref{thm:graph_construction} is that, starting with the complete graph, every iteration one of the edges gets deleted by appending a $1$-uniform word as shown in \Cref{thm:remove_edge}. Unfortunately, it outputs words of length $|\Sigma|\cdot|\overline{E}|$ but it works for all graphs. In the following, we present a characterisation of graphs that are 2-word-$\pi$-representable by $1$-uniform words, namely the permutations graphs.

\begin{definition}[Permutation Graphs~\cite{GOLUMBIC2004157}]
A graph $G = ([n],E)$ for some $n \in \N$ is a \emph{permutation graph} iff a permutation $\sigma: [n] \rightarrow [n]$ exists such that $G$ is isomorphic to $G(\sigma) = ([n], E)$ with $\{i,j\} \in E \Leftrightarrow (i - j)(\sigma^{-1}(i) - \sigma^{-1}(j)) < 0$.
\end{definition}

\Cref{fig:permutationgraph_representation} shows an example for a permutation graph and its intersection model. As a next step, we show that the graphs that are representable by two $1$-uniform words are exactly the permutation graphs.

\begin{figure}
	\begin{minipage}{.5\textwidth}
		\centering
		\begin{tikzpicture}[
			dot/.style={
				circle,
				fill=black,
				inner sep=1pt,
				minimum size=0.05cm
			}
			]
			\node [dot, label=$1$, color=\colorOne] (1) {};
			\node [dot, label=$2$, right of=1, color=\colorTwo] (2) {};
			\node [dot, label=$3$, right of=2, color=\colorThree] (3) {};
			\node [dot, label=$4$, right of=3, color=\colorFour] (4) {};
			\node [dot, label=$5$, right of=4, color=\colorFive] (5) {};
			\node [dot, label=$6$, right of=5, color=\colorSix] (6) {};
			\node [dot, label=below:$4$, below of=1, color=\colorFour] (4b) {};
			\node [dot, label=below:$5$, below of=2, color=\colorFive] (5b) {};
			\node [dot, label=below:$6$, below of=3, color=\colorSix] (6b) {};
			\node [dot, label=below:$1$, color=\colorOne, below of=4] (1b) {};
			\node [dot, label=below:$2$, below of=5, color=\colorTwo] (2b) {};
			\node [dot, label=below:$3$, below of=6, color=\colorThree] (3b) {};
			\draw [color=\colorOne](1) -- (1b);
			\draw [color=\colorTwo](2) -- (2b);
			\draw [color=\colorThree](3) -- (3b);
			\draw [color=\colorFour](4) -- (4b);
			\draw [color=\colorFive](5) -- (5b);
			\draw [\colorSix](6) -- (6b);
		\end{tikzpicture}
	\end{minipage}%
	\begin{minipage}{.5\textwidth}
		\centering \begin{tikzpicture}
			\node [color=\colorOne](1) {$1$};
			\node [color=\colorTwo][below of=1] (2) {$2$};
			\node [color=\colorThree][below of=2] (3) {$3$};
			\node [color=\colorFour][right of=1] (4) {$4$};
			\node [color=\colorFive][below of=4] (5) {$5$};
			\node [color=\colorSix][below of=5] (6) {$6$};
			\draw (1) -- (4);
			\draw (1) -- (5);
			\draw (1) -- (6);
			\draw (2) -- (4);
			\draw (2) -- (5);
			\draw (2) -- (6);
			\draw (3) -- (4);
			\draw (3) -- (5);
			\draw (3) -- (6);
		\end{tikzpicture}
	\end{minipage}
	\caption{Intersection model for the permutation graph of the permutation $\pi=(456123)$.}
	\label{fig:permutationgraph_representation}
\end{figure}
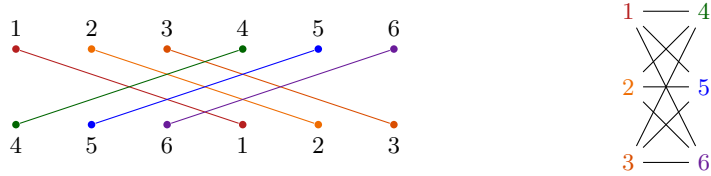

\ifpaper
\else
\begin{theorem}\label{1uniformpermutation}
	The set of graphs that are $2$-word-$\pi$-representable by $1$-uniform words is the class of permutation graphs. 
\end{theorem}

\begin{proof}
	We show, that set of graphs that are $2$-word-$\pi$-representable by $1$-uniform words are exactly the permutation graphs. 
	
	First, let $w$ and $v$ be 1-uniform words such that $\alph(w) = \alph(v)$. W.l.o.g. assume $\alph(w) = [n]$ for some $n\in\N$ and $w = 1 \cdots n$ (otherwise, $w$ is a permutation on $[n]$ and we may use the composition in the following). Let $\sigma: [n] \rightarrow [n]$ such that 
	$\sigma(i)=v[n-i+1]$. Since $v$ is 1-uniform and $\alph(v) = [n]$, $\sigma$ is a permutation. Let $i,j \in [n]$ such that $i \neq j$. Thus, we have
	\begin{align*}
		\{i,j\} \in E(G(w,v)) 
		& \Leftrightarrow \proj{i}{j}(w)=\proj{i}{j}(v) \\
		& \Leftrightarrow \proj{i}{j}(w) \neq \proj{i}{j}(v^R) \\
		& \Leftrightarrow (i < j \wedge \sigma^{-1}(i) > \sigma^{-1}(j)) \vee (i > j \wedge \sigma^{-1}(i) < \sigma^{-1}(j)) \\
		& \Leftrightarrow (i - j)(\sigma^{-1}(i) - \sigma^{-1}(j)) < 0 \\
		& \Leftrightarrow \{i,j\} \in E(G(\sigma)).
	\end{align*}
	Hence $G(w,v) = G(\sigma)$. 
	
	For the other direction, let $n\in\N$, $\sigma:[n] \rightarrow [n]$ be a permutation. Let $w  = 1 \cdots n$ and $v = (\sigma(1) \dots \sigma(n))^R$. The word $w$ is clearly 1-uniform. The word $v$ is 1-uniform, because $\sigma$ is a permutation. We prove $G(w,v) = G(\sigma)$ analogous to the first direction. \qed
\end{proof}

\fi

From the proof of \Cref{1uniformpermutation}, we can conclude that the reverse of the one-line notation of the graph's permutation 
yields a 2-word-$\pi$-representation of the graph. 

\begin{corollary}\label{cor:2WordPiRepOfPermutationGraph}
A permutation graph $G$ given by the permutation $\sigma: [n] \to [n]$ is $2$-word-$\pi$-represented by the word $1 \cdots n$ and the reverse of the one-line notation of $\sigma$ interpreted as a word.
\end{corollary}

As an example for \Cref{cor:2WordPiRepOfPermutationGraph} consider the permutation graph illustrated in \Cref{fig:permutationgraph_representation}. This graph is given by the permutation $(456123)$ and it equals $G(123456, (456123)^R) = G(123456, 321654)$.
 
\begin{remark}\label{rem:permutation-graphs-complement}
It is folklore that permutations graphs are closed under graph complement \cite{GOLUMBIC2004157}. For two 1-uniform words $w$ and $v$ such that $\alph(w) = \al(v)$, $\overline{G(w,v)} = G(w,v^R) = G(w^R,v)$. 
\end{remark} 

\ifpaper
The next result follows by the facts that a graph $G$ is a permutation graph iff $G$ and its complement are transitive orientable~\cite{PnuLemEve71} and that a graph is transitive orientable iff it does not contain an induced cycle of odd length~\cite{Gal67}.
\fi
\ifpaper
\else
\begin{corollary}\label{cor:C>5forbiddenSubgraph}
Cycles of length at least five are forbidden induced subgraphs in permutation graphs. 
\end{corollary}
\begin{proof}
Note that the permutation graphs are the intersection of the comparability graph and the co-comparability graph~\cite[Thm. 3]{PnuLemEve71}. $C_5$ is not a permutation graph, because $C_5$ is not a comparability graph. For $n \geq 6$, $C_n$ any three independent vertices in the graph are a asteroidal triple. Therefore, $C_n$ is not a co-comparability graph~\cite[Thm. 1.20]{Gal67} and thus not a permutation graph.\qed
\end{proof}
\fi

Observe that every graph with at most four nodes is a permutation graph. 
Aiming for a shorter length of the two representing words, we continue with some helpful lemmata on edge deletion and manipulation of the representing words. Note that the following two results also hold on non-$1$-uniform words. 

\ifpaper
\else
\begin{lemma}\label{thm:switche_consecutive_letters}
    Let $w,v\in\Sigma^\ast$ with $\alph(w)=\alph(v)$, 
    $\proj{\ta}{\tb}(w)=\proj{\ta}{\tb}(v)$, and $w=x\ta\tb y$ for some $x,y\in\Sigma^{\ast}, \ta, \tb \in \Sigma$. Let $w'=x\tb\ta y$.
    Then $E(G(w,v))\setminus\{\ta,\tb\}=E(G(w',v))$.
\end{lemma}

\begin{proof}
    Let $\tc\in\Sigma$ and $\td\in\Sigma\backslash\{\ta,\tb\}$. Then by $\td\neq\ta,\tb$ we have
    \[
    \proj{\tc}{\td}(w)=\proj{\tc}{\td}(x)\proj{\tc}{\td}(\ta\tb)\proj{\tc}{\td}(y)=
    \proj{\tc}{\td}(x)\proj{\tc}{\td}(\tb\ta)\proj{\tc}{\td}(y)=\proj{\tc}{\td}(w').
    \]
    On the other hand, we have $\proj{\ta}{\tb}(w)\neq\proj{\ta}{\tb}(w')$ which proves the claim.\qed
\end{proof}

\fi

The next two results show how two letters in the words can be {\em swapped} in order to get them consecutive such that 
\Cref{thm:switche_consecutive_letters} is applicable.

\ifpaper
\else
\begin{lemma}\label{thm:switche_conscutive_lettersubwords_in_both_words}
    Let $w,v\in\Sigma^\ast$ with $\alph(w)=\alph(v)$ and the factorisations
    \[
    w=\alpha_1x\ta \alpha_2 \quad \mbox{and} \quad v=\beta_1y\ta \beta_2
    \]
    with $\alpha_1,\alpha_2,\beta_1,\beta_2,x,y\in\Sigma^{\ast}$, $\ta\in\Sigma$, such that $|y|_\tb=|x|_\tb$, 
    $|\alpha_1|_\tb = |\beta_1|_\tb$ and $|\alpha_2|_\tb = |\beta_2|_\tb$ for all $\tb \in \Sigma$.
    Set $w'=\alpha_1\ta x\alpha_2$ and $v'=\beta_1\ta y\beta_2$. Then $G(w,v)=G(w',v')$.
\end{lemma}

\begin{proof}
    For all $\tc,\td\in\Sigma\backslash\{\ta\}$, we have $ \proj{\tc}{\td}(w)=\proj{\tc}{\td}(w')$ and $\proj{\tc}{\td}(v)=\proj{\tc}{\td}(v')$ which implies
    \[
    \proj{\tc}{\td}(w)=\proj{\tc}{\td}(v)\quad\mbox{iff}\quad\proj{\tc}{\td}(w')=\proj{\tc}{\td}(v')
    \]
    and thus, the edges in $G(w,v)$ and $G(w',v')$ that are not adjacent with $\ta$ are the same. Let $\tb\in\Sigma\backslash\{\ta\}$.
    
    Now consider the edges adjacent to $\ta$. Let $\tb \in \Sigma\setminus\{\ta\}$ with $\{\ta, \tb\} \in E(G(w,v))$.
    Since $\proj{\ta}{\tb}(\ta) = \ta$, we have 
    $\pi_{\ta,\tb}(w') = \proj{\ta}{\tb}(\alpha_1)\ta \proj{\ta}{\tb}(x)\proj{\ta}{\tb}(\alpha_2)$
    and 
    $\pi_{\ta,\tb}(v') = \proj{\ta}{\tb}(\beta_1)\ta \proj{\ta}{\tb}(y)\proj{\ta}{\tb}(\beta_2)$.   
    We have 
    \begin{align*}
    \proj{\ta}{\tb}(w)=& \proj{\ta}{\tb}(\alpha_1)\proj{\ta}{\tb}(x)\ta \proj{\ta}{\tb}(\alpha_2),\\ 
    \proj{\ta}{\tb}(v)=& \proj{\ta}{\tb}(\beta_1)\proj{\ta}{\tb}(y)\ta\proj{\ta}{\tb}(\beta_2).
    \end{align*}
    The condition $\{\ta, \tb\} \in E(G(w,v))$ implies $\proj{\ta}{\tb}(w) = \proj{\ta}{\tb}(v)$.
    Since $|\alpha_1|_\tc = |\beta_1|_\tc$, $|\alpha_2|_\tc = |\beta_2|_\tc$ and $|y|_\tc=|x|_\tc$ for all $\tc \in \Sigma$, we obtain $\proj{\ta}{\tb}(\alpha_1) = \proj{\ta}{\tb}(\beta_1)$, $\proj{\ta}{\tb}(x) = \proj{\ta}{\tb}(y)$ and $\proj{\ta}{\tb}(\alpha_2) = \proj{\ta}{\tb}(\beta_2)$. Thus 
    \begin{align*}
    \proj{\ta}{\tb}(w')=& \proj{\ta}{\tb}(\alpha_1)\ta\proj{\ta}{\tb}(x) \proj{\ta}{\tb}(\alpha_2)\\
                       =& \proj{\ta}{\tb}(\beta_1)\ta\proj{\ta}{\tb}(y) \proj{\ta}{\tb}(\beta_2)\\
                       =& \proj{\ta}{\tb}(v')
    \end{align*}
    and $\{\ta, \tb\} \in E(G(w',v'))$.

    At last, consider all $\tb \in \Sigma\setminus \{\ta\}$ with $\{\ta, \tb\} \notin E(G(w,v))$. Hence $\proj{\ta}{\tb}(w) \neq \proj{\ta}{\tb}(v)$. Since $|\alpha_1|_\tc = |\beta_1|_\tc$, $|\alpha_2|_\tc = |\beta_2|_\tc$ and $|y|_\tc=|x|_\tc$ for all $\tc \in \Sigma$, we obtain $\proj{\ta}{\tb}(\alpha_1) \neq \proj{\ta}{\tb}(\beta_1)$, $\proj{\ta}{\tb}(x) \neq \proj{\ta}{\tb}(y)$ or $\proj{\ta}{\tb}(\alpha_2) \neq \proj{\ta}{\tb}(\beta_2)$. \\
Case 1: First consider the case that  $\proj{\ta}{\tb}(\alpha_1) \neq \proj{\ta}{\tb}(\beta_1)$ or $\proj{\ta}{\tb}(\alpha_2) \neq \proj{\ta}{\tb}(\beta_2)$. In this case we get $\proj{\ta}{\tb}(w') \neq \proj{\ta}{\tb}(v')$. 
Hence $\{\ta, \tb\} \notin E(G(w',v'))$. \\
Case 2: Now consider the case $\proj{\ta}{\tb}(\alpha_1) = \proj{\ta}{\tb}(\beta_1)$ and $\proj{\ta}{\tb}(\alpha_2) = \proj{\ta}{\tb}(\beta_2)$. Hence $\proj{\ta}{\tb}(x) \neq \proj{\ta}{\tb}(y)$. We have 
    \begin{align*}
    \proj{\ta}{\tb}(w')=& \proj{\ta}{\tb}(\alpha_1)\ta \proj{\ta}{\tb}(x) \proj{\ta}{\tb}(\alpha_2),\\ 
    \proj{\ta}{\tb}(v')=& \proj{\ta}{\tb}(\beta_1)\ta \proj{\ta}{\tb}(y)\proj{\ta}{\tb}(\beta_2) \\
                       =& \proj{\ta}{\tb}(\alpha_1)\ta \proj{\ta}{\tb}(y)\proj{\ta}{\tb}(\alpha_2).
    \end{align*}
    We have $\proj{\ta}{\tb}(w') \neq \proj{\ta}{\tb}(v')$, because of $\proj{\ta}{\tb}(x) \neq \proj{\ta}{\tb}(y)$. Therefore, $\{\ta, \tb\} \notin E(G(w',v'))$.
    This concludes the proof.\qed
\end{proof}
\fi

\begin{remark}\label{thm:switche_lettersubwords_counterdirectional_in_both_words}
    Similar to \Cref{thm:switche_conscutive_lettersubwords_in_both_words} the following can be proven for $1$-uniform words:
    Let $w,v\in\Sigma^\ast$ with $|w|_\tb=|v|_\tb=1$ for all $\tb\in\al(w)\cup\al(v)$ and $w=\alpha_1x\ta\alpha_2$, $v=\beta_1\ta y\beta_2$ with $\alpha_1,\alpha_2,\beta_1,\beta_2,x,y\in\Sigma^\ast$ and $\ta\in\Sigma$ such that $|y|_\tb=|x|_\tb,~|\alpha_1|_\tb=|\beta_1|_\tb$ and $|\alpha_2|_\tb=|\beta_2|_\tb$ for all $\tb\in\Sigma$.
    Set $w'=\alpha_1x\ta \alpha_2$ and $v'=\beta_1\ta y\beta_2$. Then $G(w,v)=G(w',v')$.
\end{remark}

The full characterisation (\Cref{1uniformpermutation}) of $1$-uniform $2$-word-$\pi$-representability and the fact that all graphs are $2$-word-$\pi$-representable (\Cref{thm:graph_construction}) motivate the following definition which gives a fine-grained classification of all $2$-word-$\pi$-representable graphs. Our goal it to investigate structural properties of this hierarchy.

\begin{definition}
	For each $k \in \mathbb{N}$, the set 
	$$\mathcal{G}_k := \{ G \mid \exists w,v \in V(G)^{\ast}: \text{$w$ and $v$ are $k$-uniform and } G = G(w,v)\}$$ 
	is the set of all graphs that can be represented by two $k$-uniform words. 
\end{definition}

\begin{remark} 
For each $k \in \mathbb{N}$, $\mathcal{G}_k$ is hereditary. 
\end{remark}

\begin{remark} For each $k \in \mathbb{N}$, $\mathcal{G}_k \subseteq \mathcal{G}_{k+1}$.
	For justification consider, $w,v \in \Sigma^*$ both $k$-uniform that represent $G(w,v) \in \mathcal{G}_k$.  Let $u \in \al(w)^*$ be a $1$-uniform word. By \Cref{thm:append_word} we get that $G(w,v) = G(wu,vu)$. Since $wu, vu$ are $(k+1)$-uniform, we obtain that $G(w,v) \in \mathcal{G}_{k+1}$.
\end{remark}

The following remark is an immediate consequence of \Cref{thm:graph_construction}.
\begin{remark} 
$\bigcup_{k = 1}^{\infty} \mathcal{G}_k$ is the class of all graphs. 
\end{remark}

Graphs that are not permutation graphs are not $2$-word-$\pi$-representable by two $1$-uniform words. However those where only one edge is missing to be a permutation graph, are $2$-word-$\pi$-representable by two $2$-uniform words.

\ifpaper
\else
\begin{theorem}\label{thm:1-nearly-permutationgraph}
    For a permutation graph $G'$ and a spanning subgraph $G$ of $G'$ with $|E(G')\setminus E(G)|=1$, $G$ is $2$-word-$\pi$-representable by two $2$-uniform words.
\end{theorem}

\begin{proof}
    Let $G$ be a graph and $G'$ a permutation graph such that $|E(G')\setminus E(G)|=1$. Then there exist two $1$-uniform words $w,v$ with $G'=G(w,v)$. Therefore $G=G(wu\ta\tb,vu\tb\ta)$ with $\{\ta,\tb\}\in E(G')\setminus E(G)$, $\al(u)=\al(w)\setminus\{\ta,\tb\}$ and $|u|=|\al(w)|-2$.\qed
\end{proof}
\fi

The approach here is to delete only independent edges by swapping the corresponding two letters. To do this, however, it is necessary to determine
which edges are independent, which is an NP-hard problem.

\ifpaper
\else
\begin{theorem}\label{thm:deletion_independent_edges}
	Let $w,v\in\Sigma^\ast$ 
	and $\IE\subseteq E(G(w,v))$ a set of independent edges. 
	There exist 1-uniform words $u_w,u_v\in\Sigma^\ast$ with $\alph(u_w)=\alph(u_v)=\alph(w)$ such that $E(G(wu_w,vu_v))=E(G(w,v))\setminus \IE$.
\end{theorem}
\begin{proof}
	Define $u\in (V(G(w,v))\setminus \bigcup \IE)^*$, such that $|u|_\ta=1$ for all $\ta\in V(G(w,v))\setminus \bigcup \IE$.
	We construct $u_w, u_v$ the following way: Let $\alph(u_w)=\alph(u_v)=\bigcup \IE$, $|u_w|=|u_v|=|\bigcup \IE|$. Define $u_w = \prod_{\{\ta,\tb\} \in \IE} \ta \cdot \tb$. Further, construct $u_v$ for all odd $i \in [|u_w|]$ by $u_v[i+1] = u_w[i]$ and $u_v[i] = u_w[i+1]$. 
	Then we obtain $\pi_{\ta,\tb}(u_w)=\pi_{\ta,\tb}(u_v)^R$ for all $\{\ta,\tb\}\in \IE$ and $\pi_{\ta,\tb}(u_w)=\pi_{\ta,\tb}(u_v)$ for all $\{\ta,\tb\}\not\in \IE$.
	Set $w'=wuu_w$ and $v'=vuu_v$. Thus for all $\ta,\tb\in\alph(w)$ the edge $\{\ta,\tb\}$ is not in $E(G(w',v'))$ iff $\pi_{\ta,\tb}(w')\neq \pi_{\ta,\tb}(v')$. The same holds for the edges of $G(w,v)$. Since for all $\{\ta,\tb\}\in \IE$, we have $u_w=x \pi_{\ta,\tb}(u_w) y=x \pi_{\ta,\tb}(u_v)^R y$ with $x,y\in(\alph(w)\setminus\{\ta,\tb\})^\ast$, $V(G(w',v'))=V(G(w,v))$ and $E(G(w',v'))=E(G(w,v))\setminus \IE$.\qed
\end{proof}
\fi

\Cref{thm:deletion_independent_edges} can now be used instead of \Cref{thm:remove_edge} in the construction of words $w,v$, such that $G(w,v)=G$ for a given graph $G=(V,E)$, similar to \Cref{thm:graph_construction}. 

Furthermore, \Cref{thm:deletion_independent_edges} led us to the problem of 
determining independent edges. This can be achieved with the help of edge colouring. All edges that have been given the same colour do not have a common vertex and can therefore be deleted at the same time.
In this part, we present more insights on independent edges sets and edge colourings. We start with the necessary definition.

\begin{definition}
A {\em valid edge colouring} of a graph $G=(V,E)$ is a function $f:E\rightarrow[k]$ for some $k\in\N$ such that
$f(e_1)\neq f(e_2)$ if $e_1\cap e_1\neq \emptyset$ for all $e_1,e_2\in E$. The minimal possible $k$ such that
$f:E\rightarrow [k]$ is a valid edge-colouring is called the {\em chromatic index} and denoted by $\chi'(G)$. It is folklore by Vizing's theorem that we have $\chi'(G)\in\{\Delta(G),\Delta(G)+1\}$. 
\end{definition}

An algorithm for computing the independent edges for a graph with a given edge colouring is given by:

\begin{algorithm}
	\DontPrintSemicolon
	\SetKwInOut{Input}{input}\SetKwInOut{Output}{output}
	\Input{Edge colouring $f:\overline{E}\rightarrow \mathbb{N}$, chromatic index $\chi(\overline{G})$ and $\overline{E}$}
	\Output{A set of sets of independent edges}
	\tcc{declare $\IES$ as set of $\chi(\overline{G})$ sets}
	$\IES\longleftarrow [\chi(\overline{G})]$\;
	\ForEach{$\{\ta,\tb\}\in \overline{E}$}{
		$\IES[f(\{\ta,\tb\})]$ add $\{\ta,\tb\}$\;
	}
	\KwRet{$\IES$}\;
	\caption{Edge colouring to set of sets of independent edge.}
	\label{algo:colouring_to_set}
\end{algorithm}

\ifpaper
\else
\begin{theorem}\label{thm:graph_construction_independent_edges}
	For every graph $G=(V,E)$ exist $w,v\in\Sigma^\ast$ with $\alph(w)=\alph(v)=V$ and a set of sets of independent edges $\IES$ with $\bigcup \IES=\overline{E}$ and $M\cap N=\emptyset$ for all $M,N\in \IES$, such that $G(w,v)=G$ with $|w|=|v|=|V|\cdot\chi(G)$.
\end{theorem}
\begin{proof}\label{proof:graph_construction_independent_edges}
	The set of sets with the independent edges $\IES$ is received by applying \Cref{algo:colouring_to_set} to a minimal edge colouring of the graph $(V,\overline{E})$. 
	The words $w,v$ are constructed iteratively by removing one set of independent edges $M\in \IES$ after another with \Cref{thm:deletion_independent_edges}, starting with $w=v=\emptyWord$. In each iteration one set of independent edges is removed. After removing all edges $\{\ta,\tb\}\in M$ in every set $M\in \IES$, only the edges in $E$ are left. This takes $|\IES|=\chi(\overline{G})$ iterations and since in each iteration the length of words $w,v$ is increased by $|V|$, $|w|=|v|=|V|\cdot\chi(G)$.\qed
\end{proof}
\fi

\ifpaper
\else
\begin{corollary}
	The complexity of \Cref{algo:colouring_to_set} is in $\mathcal{O}\left(|\bigcup \IES|\right)$ since the for-loop is executed for each edge to be deleted.
\end{corollary}
\fi

\begin{example}\label{example:graph_construction_independent_edges}
	Given $G=(V,E)$ with $V=\{a,b,c,d\}$, $E=\{\{a,b\},\{a,d\},\{c,d\}\}$. For the set of sets of independent edges $\IES=\{\{\{\ta,\tc\},\{\tb,\td\}\},\{\{\tb,\tc\}\}\}$ the words $w,v$ can be constructed in two instead of three iterations.
	\begin{table}[h]
        \centering
		\begin{tabular}{c|c|c|c}
			$w$										& $v$										& delete						& $E$\\
			\hline
			$\emptyWord$							& $\emptyWord$								& -								& $\emptyset$\\	
			$\ta\tc\tb\td$							& $\tc\ta\td\tc$							& $\{\{\ta,\tc\},\{\tb,\td\}\}$	& $\{\{\ta,\tb\},\{\ta,\td\},\{\tb,\tc\},\{\tc,\td\}\}$\\
			$\ta\tc\tb\td\ta\td\tb\tc$				& $\tc\ta\td\tc\ta\td\tc\tb$				& $\{\{\tb,\tc\}\}$				& $\{\{\ta,\tb\},\{\ta,\td\},\{\tc,\td\}\}$
		\end{tabular}
		\caption{Construction of $w,v$ with \Cref{thm:graph_construction_independent_edges}}
		\label{table:example_graph_construction_set_independent_edges}
	\end{table}
	\begin{figure}
		\centering
		\begin{minipage}{.33\textwidth}
			\centering
			\begin{tikzpicture}
				\node (a) {$\ta$};
				\node[right of=a] (b) {$\tb$};
				\node[below of=a] (c) {$\tc$};
				\node[below of=b] (d) {$\td$};
			\end{tikzpicture}

			Initial Graph
		\end{minipage}%
		\begin{minipage}{.33\textwidth}
			\centering
			\begin{tikzpicture}
				\node (a) {$\ta$};
				\node[right of=a] (b) {$\tb$};
				\node[below of=a] (c) {$\tc$};
				\node[below of=b] (d) {$\td$};
				\draw (a) -- (b);
				\draw (a) -- (d);
				\draw (b) -- (c);
				\draw (c) -- (d);
			\end{tikzpicture}

			First iteration
		\end{minipage}%
		\begin{minipage}{.33\textwidth}
			\centering
			\begin{tikzpicture}
				\node (a) {$\ta$};
				\node[right of=a] (b) {$\tb$};
				\node[below of=a] (c) {$\tc$};
				\node[below of=b] (d) {$\td$};
				\draw (a) -- (b);
				\draw (a) -- (d);
				\draw (c) -- (d);
			\end{tikzpicture}

			Second iteration
		\end{minipage}
		\caption{Graphs for construction as in \Cref{table:example_graph_construction_set_independent_edges}}
	\end{figure}
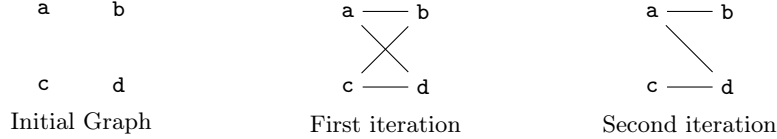
\end{example}

\begin{algorithm}
	\DontPrintSemicolon
	\SetKwInOut{Input}{input}\SetKwInOut{Output}{output}
	\Input{A set of sets $\IES$ containing edges $\{\ta,\tb\}$}
	\Output{True if edges in sets are independent, otherwise false}
	$A\longleftarrow \emptyset$\;
	\ForEach{$M\in \IES$}{
		\ForEach{$\{\ta,\tb\}\in M$}{
			\If{$\{\ta,\tb\}\in A$}{
				\KwRet{False}\;
			}
			$A\longleftarrow A\cup\{\{\ta,\tb\}\}$\;
		}
	}
	\ForEach{$M\in \IES$}{
		\ForEach{$\{\ta,\tb\}\in M$}{
			$M\longleftarrow M\setminus\{\{\ta,\tb\}\}$\;
			\ForEach{$\{\tc,\td\}\in M$}{
				\If{$\{\ta,\tb\}\cap\{\tc,\td\}\neq\emptyset$}{
					\KwRet{False}\;
				}
			}
		}
	}
	\KwRet{True}\;
	\caption{Verifier for set of sets of independent edges}
	\label{algo:npc_verifier}
\end{algorithm}

\ifpaper
\begin{theorem}\label{thm:independent_edges_NPC}
	Given a graph $G=(V,E)$, the determination of a set of sets of independent edges $\IES$ with $\bigcup \IES=\overline{E}$, $\{\ta,\tb\}\in\bigcup \IES$ for all $\{\ta,\tb\}\in\overline{E}$ and $|\IES|=\chi(\overline{G})$ is in NPC.
\end{theorem}
\begin{proof}
	Use of \Cref{algo:npc_verifier} as a verifier. Since a given instance $\IES$ has at most $|\bigcup \IES|=|\overline{E}|=\frac{|V|(|V|-1)}{2}$ edges in case of a full graph, the complexity of the first nested for loop is in $\mathcal{O}\left(\frac{|V|(|V|-1)}{2}\right)$ and so in $\mathcal{O}(|V|^2)$. Furthermore the complexity of the second nested loop is in $\mathcal{O}\left(\frac{\frac{|V|(|V|-1)}{2}(\frac{|V|(|V|-1)}{2}-1)}{2}\right)$ and so in $\mathcal{O}(|V|^4)$. Thus the complexity of the verifier is polynomial in the number of vertices of the given graph $G$. Therefore, the determination of a set of sets of independent edges is in NP.\\
	Let $f\in$ minimal edge colouring of $\overline{G}$ for a given $G=(V,E)$. Then the $\IES$ obtained by \Cref{algo:colouring_to_set} applying to $f$, the chromatic index $\chi(\overline{G})$ and $\overline{E}$ has $|\IES|=\chi(\overline{G})$ and for all sets $M\in \IES$ the elements of $M$ are mapped to the same number/colour and therefore, are independent to each other. Furthermore $\bigcup \IES=\overline{E}$ since $f$ maps all elements $\{\ta,\tb\}\in\overline{E}$ to $\{1,2,\dots,\chi(\overline{G})\}$, thus for all edges $\{\ta,\tb\}\in\overline{E}$ exists $M\in \IES$, such that $\{\ta,\tb\}\in M$. This means that a minimal edge colouring instance can be transformed into a set of sets of independent edges, of which the big union is equal to $\overline{E}$, in polynomial time.\\
	Let $\IES$ a set of sets of independent edges with $\bigcup \IES=\overline{E}$ and $|\IES|=\chi(\overline{G})$. Assign indices to the sets of $\IES$ so that $\IES=\{M_1,M_2,\dots,M_{\chi(\overline{G})}\}$ and construct a function $f:\bigcup \IES\rightarrow \{1,2,\dots, |\IES|\}$ which maps all $\{\ta,\tb\}\in\bigcup \IES$ to $i$ with $\{\ta,\tb\}\in M_i$ and $i\in\{1,2,\dots,\chi(\overline{G})\}$. Then $f$ is a minimum edge colouring for $\overline{G}$ and it is shown that a set of sets of independent edges which big union is equal to $\overline{E}$ can be transformed to a instance of minimal edge colouring of the graph $(V,\overline{E})$.\\
	Since minimal edge colouring reduces to set of sets of independent edges and it has been shown that constructing a set of sets of independent edges of given edges is in NP and edge colouring is in NPC \cite{Hol81_2}, we conclude that this problem is NP-complete.\qed
\end{proof}
\else
\begin{theorem}\label{thm:independent_edges_NPC}
	Given a graph $G=(V,E)$, the determination of a set of sets of independent edges $\IES$ with $\bigcup \IES=\overline{E}$, $\{\ta,\tb\}\in\bigcup \IES$ for all $\{\ta,\tb\}\in\overline{E}$ and $|\IES|=\chi(\overline{G})$ is in NPC.
\end{theorem}

\fi

As a next step we show that cycle graphs of length at least five can be $2$-word-$\pi$-represented by two $2$-uniform words. To this end we combine the ideas of \Cref{thm:1-nearly-permutationgraph,thm:deletion_independent_edges}.

\Cref{fig:permutationgraph_representation} shows an example for a permutation graph and its intersection model. As a next step, we show that the graphs that are representable by two $1$-uniform words are exactly the permutation graphs.

\begin{figure}
\begin{minipage}{0.5\textwidth}
		\centering
		\begin{tikzpicture}[
			dot/.style={
				circle,
				fill=black,
				inner sep=1pt,
				minimum size=0.05cm
			}]
			
			\draw [color=gray] (-0.2,0) -- (2.6,0);
			\draw [color=gray] (-0.2,1) -- (2.6,1);

			\node [dot, color=\colorOne, label={below:2}] (2) at (0,0) {};
			\node [dot, color=\colorOne, label={below:1}] (1) at (1,0) {};
			\node [dot, color=\colorOne, label={above:1}] (1b) at (0,1) {};
			\node [dot, color=\colorOne, label={above:2}] (2b) at (1,1) {};
			\draw [color=\colorOne](1) -- (1b);
			\draw [color=\colorOne](2) -- (2b);

			\node [dot, color=\colorFive, label={below:3}] (3) at (0.7,0) {};
			\node [dot, color=\colorFive, label={below:4}] (4) at (1.7,0) {};
			\node [dot, color=\colorFive, label={above:4}] (4b) at (0.7,1) {};
			\node [dot, color=\colorFive, label={above:3}] (3b) at (1.7,1) {};
			\draw [color=\colorFive](3) -- (3b);
			\draw [color=\colorFive](4) -- (4b);

			\node [dot, color=\colorFour, label={below:6}] (6) at (1.4,0) {};
			\node [dot, color=\colorFour, label={below:5}] (5) at (2.4,0) {};
			\node [dot, color=\colorFour, label={above:5}] (5b) at (1.4,1) {};
			\node [dot, color=\colorFour, label={above:6}] (6b) at (2.4,1) {};
			\draw [color=\colorFour](5) -- (5b);
			\draw [color=\colorFour](6) -- (6b);

			\node [dot, color=\colorThree, label={below:7}] (7) at (2.1,0) {};
			\node [dot, color=\colorThree, label={above:7}] (7b) at (2.1,1) {};
			\draw [color=\colorThree](7) -- (7b);
		\end{tikzpicture}
	\end{minipage}%
	\begin{minipage}{0.5\textwidth}
		\centering
		\begin{tikzpicture}[
			dot/.style={
				circle,
				fill=black,
				inner sep=1pt,
				minimum size=0.05cm
			}]
			
			\node (1) at (0,0) {\color{\colorOne}1};
			\node (2) at (1,0) {\color{\colorOne}2};
			\draw [color=\colorOne](1) -- (2);

			\node (3) at (0,1) {\color{\colorFive}3};
			\node (4) at (1,1) {\color{\colorFive}4};
			\draw (1) -- (3);
			\draw (2) -- (4);
			\draw [color=\colorFive](3) -- (4);

			\node (5) at (0,2) {\color{\colorFour}5};
			\node (6) at (1,2) {\color{\colorFour}6};
			\draw (3) -- (5);
			\draw (4) -- (6);
			\draw (5)[color=\colorFour] -- (6);
			
			\node (7) at (0.5,2.5) {\color{\colorThree}7};
			\draw (5) -- (7);
			\draw (6) -- (7);
		\end{tikzpicture}
	\end{minipage}
\caption{The intersection model (on the left side) shows that the graph (on the right side) constructed from a ladder graph by adding an additional vertex connected to the vertices of the upper step is a permutation graph.}
\label{fig:2uniformcycles}
\end{figure}
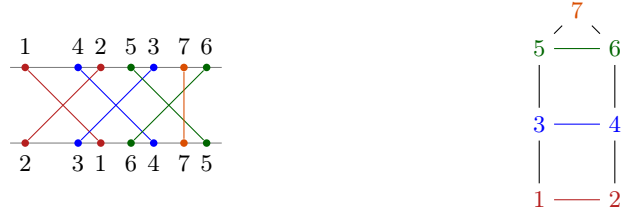

\ifpaper
\else
\begin{theorem}\label{thm:2uniformcycles}
    The set of cycle graphs with $|V|>4$ is contained in $\mathcal{G}_2$.
\end{theorem}
\begin{proof}\label{thm:2uniformcycles}
Let $n > 4$ and $L_n := ([n],E_n)$ with $E_n := \{ \{ 2i-1,  2i\} \mid i \in [\lfloor \frac{n}{2}\rfloor]\} \cup \{ \{2i-1,2i+1\}, \{2i,2i+2\} \mid i \in [\lfloor \frac{n}{2}\rfloor -1]\} \cup \{ \{n-2, n\}, \{n-1, n\}\}$. $L_n$ is a permutation graph (confer \Cref{fig:2uniformcycles}) and therefore in $\mathcal{G}_1$ by \Cref{1uniformpermutation}. $I_n := \{ \{ 2i-1,  2i\} \mid i \in [\lfloor \frac{n}{2}\rfloor] \setminus \{1\} \}$ is an independent edge set in $L_n$. Note that the graphs $([n], E_n \setminus I_n)$ are exactly the cycle graphs of length at least five. By \Cref{thm:deletion_independent_edges}, the edges in $I_n$ can be deleted from $L_n$ by appending $1$-uniform words. This proves the statement.\qed
\end{proof}

\fi

Note that we also get the following by combining \Cref{cor:C>5forbiddenSubgraph}  and \Cref{thm:2uniformcycles}.
\begin{corollary}
$\mathcal{G}_1 \subsetneq \mathcal{G}_2$. 
\end{corollary}

We want to characterize the graph class $\mathcal{G}_2$. To this end, we study subclasses of $\mathcal{G}_2$. In \cite{DBLP:journals/combinatorics/JonesKPR15} Jones et al. study the representation of graphs via pattern avoiding words, i.e.
$u$-representable graphs for pattern $u \in \{1,2\}^* \setminus \{2\}^*$ of length at least $2$. 
The reduction $red(w)$ of a word $w \in \mathbb{N}^*$ is the word that we obtain by replacing every appearance of the \nth{$i$} smallest letter in $w$ by $i$. For example $red(524) = 312$. A word has a $12$-match iff it has a factor $f$ such that $red(f) = 12$. For example $5364$ has a $12$-match at position $2$ because $red(36) = 12$. A graph $G = (V, E)$ with $V \subseteq \mathbb{N}$ is $12$-representable iff there exits a word $w$ such that $\letters(w)=V$ and for each $\{i,j\} \in \binom{V}{2}$, $\{i,j\} \in E \Leftrightarrow \pi_{i,j}(w) $ has no $12$-match. Note that the classical word-representable graphs are exactly the $11$-representable graphs.

\ifpaper
\else
\begin{theorem}
	The $12$-representable graphs are a subset of $\mathcal{G}_2$.
\end{theorem}
\begin{proof}
	Let $G$ be $12$-representable, i.e. $V(G) \subseteq \mathbb{N}$. W.l.o.g. assume $V(G) = [|V|]$. According to \cite[p. 7]{DBLP:journals/combinatorics/JonesKPR15}, 
	there exists a $2$-uniform word $v$ that $12$-represents $G$. For each $\{i,j\} \in \binom{V(G)}{2}$ such that $i < j$, $\pi_{i,j}(v)$ has no $12$-match iff $ij$ is not a factor of $v$ iff $\pi_{i,j}(w) = jjii$. Let $w := |V|^2 (|V|-1)^2 \dots 1^2$. We will show that $G(w,v) = G$. For each $\{i,j\} \in \binom{V(G)}{2}$ such that $i < j$, 
	\begin{align*}
		\{ i,j\} \in E(G(w,v)) 
		& \Leftrightarrow \pi_{i,j}(w) = \pi_{i,j}(v) \\
		& \Leftrightarrow \pi_{i,j}(w) = jjii \\
		& \Leftrightarrow \pi_{i,j}(w) \text{ has no $12$-match} \\
		& \Leftrightarrow \{ i,j\} \in E(G).
	\end{align*}
	Therefore, $G(w,v) = G$ and $G \in \mathcal{G}_2$.\qed
\end{proof}
\fi

According to \cite[Fig. 1]{DBLP:journals/combinatorics/JonesKPR15}, the co-interval graphs and the permutation graphs are a subset of the $12$-representable graphs. 

\begin{corollary}
The co-interval graphs are a subset of $\mathcal{G}_2$.
\end{corollary}

We will use a result from \cite{FenFFKS2026} to show that the hierarchy of the graph classes $\mathcal{G}_k$ does not collapse. 

\ifpaper
\else
\begin{theorem}\label{thm:Gknotall}
	For all $k \in \mathbb{N}$, $\mathcal{G}_k$ is not the class of all graphs. 
\end{theorem}
\begin{proof}
	In \cite{FenFFKS2026} $h_{\ta,\tb}: \Sigma^* \to \{0,1\}^*$ was defined as the homomorphism that maps the letter $\ta$ to $0$, the letter $\tb$ to $1$ and deletes every other letter of a given alphabet $\Sigma$. For each language $L \subseteq \{0,1\}^*$ that is closed under the complement morphism $\tilde{\cdot}: \{0,1\}^* \to \{0,1\}^*, b \mapsto 1-b$, $G(L,w)$ is the graph that is $L$-represented by the word $w$.
	
	Let $L_k = \{ u^2 \mid u \in 0^k \shuffle 1^k\}$. Clearly $\tilde{L} = L$. We will show $\mathcal{G}_k \subseteq \mathcal{G}_{L_k}$. Let $G \in \mathcal{G}_k $. There exist $k$-uniform words $w,v \in V(G)^*$ such that $G = G(w,v)$. We will show that $G(w,v) = G(L,wv)$. 
	For each $\{\ta,\tb\} \in \binom{V}{2}$, 
	$
	\pi_{\ta,\tb}(w) = \pi_{\ta,\tb}(v) 
	\Leftrightarrow h_{\ta,\tb}(w) = h_{\ta,\tb}(v) 
	\Leftrightarrow h_{\ta,\tb}(wv) = h_{\ta,\tb}(w) \cdot h_{\ta,\tb}(v) \in L_k.
	$
	The last equivalence holds because $h_{\ta,\tb}(w),h_{\ta,\tb}(v) \in 0^k \shuffle 1^k$ since they are both $k$-uniform. This concludes the proof of $\mathcal{G}_k \subseteq \mathcal{G}_{L_k}$. According to \cite[Theorem 1]{FenFFKS2026} $\mathcal{G}_{L_k}$ is not the class of all graphs because $L_k$ is finite. \qed
\end{proof}
\fi

\begin{remark}\label{rem:factorial-speed}
	By \cite[Theorem 1]{FenFFKS2026} we can deduce that no graph class with superfactorial speed $2^{\Theta(n^2)}$ can be contained completely in any class $G_k$.  In particular, we know that no class $G_k$ for $k \in \N$ can contain all bipartite, split, cobipartite, word-representable, comparability or chordal graphs.
\end{remark}

\begin{conjecture}
For each $k \in \mathbb{N}$, the inclusion $\mathcal{G}_k \subseteq \mathcal{G}_{k+1}$ is proper. 
\end{conjecture}


	\section{Operations on Graphs}\label{operations}
	In the context of graph theory, it is reasonable to analyse the standard operations on graphs and how to realise 
them on graphs as defined in \Cref{def:graph_p} and their representing words. For this purpose, the operations of join, 
the insertion of universal vertices, union,  and the insertion of isolated vertices are considered.
This line of research
was also followed for the original definition (cf. \cite{DBLP:series/eatcs/KitaevL15,Choi_2018}) where the authors investigated which operations on graphs preserve word-representability. Since with our definitions all graphs are word-representable, the question is just how the words can be altered if we have operations on the graphs.
But first, we remark that a $1$-uniform word and its reversed  $2$-word-$\pi$-represent the empty graph.

\begin{remark}\label{thm:empty_graph_reversed_words}
	Let $w\in\Sigma^\ast$ be $1$-uniform and $|\Sigma| > 1$. 
	This implies immediately $\proj{\ta}{\tb}(w)\neq\proj{\ta}{\tb}(w^R)$ for all $\ta,\tb\in\alph(w)$, and therefore $G(w,w^R)$ is empty.
\end{remark}

\ifpaper
\else
\begin{definition}\label{def:union}
	Let $w,w',v,v'\in\Sigma^\ast$ with $\alph(w)\cap\alph(w')=\emptyset$ s.t. $G(w,v)$, $G(w',v')$ are $2$-word-$\pi$-representable graphs.
	Then $G(w,v)\cup G(w',v')$ is the \textit{union} of the two graphs defined by the set of vertices $V=V(G(w,v))\cup V(G(w',v'))$ and the set of edges $E(G(w,v)\cup G(w',v'))=E(G(w,v))\cup E(G(w',v'))$.
\end{definition}
\fi

For the rest of this section let $w,w',v,v'\in\Sigma^\ast$ with $\alph(w)\cap\alph(w')=\emptyset$, such that $G(w,v)$ and $G(w',v')$ are $2$-word-$\pi$-representable.
First, we investigate the join of two graphs.

\ifpaper
\else
\begin{theorem}\label{thm:graph-join}
	If  $|w|_\ta=|v|_\ta$, $|w'|_\tb=|v'|_\tb$ for all $\ta\in\alph(w),\tb\in\alph(w')$ then the join $G(w,v) \nabla G(w',v')$ is given by $G(ww',vv')$.
\end{theorem}

\begin{proof}
	First notice 
	\begin{align*}
	V(G(w,v) \nabla G(w',v')) &= V(G(w,v)) \cup V(G(w',v')) = \al(w) \cup \al(w')\\
	&= V(G(ww',vv')).
	\end{align*}
	
	Second, we are showing that $E(G(w,v) \nabla G(w',v')) = E(G(ww',vv'))$.
	Since $\alph(w)\cap\alph(w')=\emptyset$, we have $\pi_{\ta,\tb}(w)=\pi_{\ta,\tb}(ww')$ and $\pi_{\ta,\tb}(v)=\pi_{\ta,\tb}(vv')$ for all $\ta,\tb\in\alph(w)$, and $\pi_{\tc,\td}(w')=\pi_{\tc,\td}(ww')$ and $\pi_{\tc,\td}(v')=\pi_{\tc,\td}(vv')$ for all $\tc,\td\in\alph(w')$. Thus, $E(G(w,v)) \cup E(G(w',v'))\subseteq E(G(ww',vv'))$. Since $\alph(w)\cap\alph(w')=\emptyset$ there do not exist any additional edges for $\ta,\tb\in\alph(w)$ or $\tc,\td\in\alph(w')$ in $G(ww',vv')$.
	Further, for all $\ta\in\alph(w)$ and $\tc\in\alph(w')$ we have $\pi_{\ta,\tc}(ww')=\ta^{|w|_\ta}\tc^{|w'|_\tc}$. Therefore, with $|w|_\ta=|v|_\ta$ and $|w'|_\tc=|v'|_\tc$ for all $\ta\in\alph(w)$ and all $\tc\in\alph(w')$ we obtain $\pi_{\ta,\tc}(ww')=\pi_{\ta,\tc}(vv')$ and so $\{\ta,\tc\}\in E(G(ww',vv'))$.
	Thus, we conclude $E(G(ww',vv'))=E(G(w,v))\cup E(G(w',v'))\cup \{\{u,u'\} \mid u \in V(G(w,v)), u' \in V(G(w',v'))\}$.\qed
\end{proof}

\fi

\begin{remark}
	Note that we assume for realising graph operations that $|w|_\ta = |v|_\ta, |w'|_\tb = |v'|_\tb$ for all $\ta \in \al(w), \tb \in \al(w')$ holds. This is necessary to keep control if the projections of the concatenation $ww'$ and $vv'$. For example, consider $w = \mathtt{abbc}, v = \mathtt{abc}, w' = \mathtt{de}, v' = \mathtt{ed}$. Then $G(w,v) = (\{\ta,\tb,\tc\}, \{\{\ta,\tc\}\})$ and $G(w',v') = (\{\td,\te\}, \emptyset)$. With $ww' = \mathtt{abbcde}, vv' = \mathtt{abced}$ we get $G(ww',vv') = (\{\ta,\tb,\tc,\td,\te\},\{\{\ta,\tc\}, \{\ta,\td\}, \{\ta,\td\}, \{\ta,\te\}, \{\tc,\td\}, \{\tc,\te\}\})$. Nevertheless, a join of $G(w,v)$ and $G(w',v')$ must contain the edges $\{\tb,\td\}, \{\tb,\te\}$ by definition which is not the case since $|w|_\tb \neq |v|_\tb$.
\end{remark}

Let $w=\ta\td\tb\tc$, $v=\tb\ta\tc\td$, $w'=\te\tg\mathtt{h}\tf$ and $v'=\mathtt{h}\tg\te\tf$. Then the joined graph \Cref{graph:example_join_joined} of the graphs \Cref{graph:example_join_1,graph:example_join_2} is represented by the words $ww'=\ta\td\tb\tc\te\tg\mathtt{h}\tf$ and $vv'=\tb\ta\tc\td\mathtt{h}\tg\te\tf$.

\begin{figure}
	\begin{minipage}{0.5\textwidth}
		\centering
		\begin{tikzpicture}
			\node (a) {$\ta$};
			\node[right of=a] (b) {$\tb$};
			\node[below of=a] (c) {$\tc$};
			\node[right of=c] (d) {$\td$};
			\draw[dotted, line width=1.0pt] (a) -- (d);
			\draw[dotted, line width=1.0pt] (a) -- (c);
			\draw[dotted, line width=1.0pt] (b) -- (c);
		\end{tikzpicture}
		\caption{Graph $G(w,v)$}
		\label{graph:example_join_1}
	\end{minipage}%
	\begin{minipage}{0.5\textwidth}
		\centering
		\begin{tikzpicture}
			\node (e) {$\te$};
			\node[right of=e] (f) {$\tf$};
			\node[below of=e] (g) {$\tg$};
			\node[right of=g] (h) {$\mathtt{h}$};
			\draw[dashed, line width=1.0pt] (e) -- (f);
			\draw[dashed, line width=1.0pt] (f) -- (g);
			\draw[dashed, line width=1.0pt] (f) -- (h);
		\end{tikzpicture}
		\caption{Graph $G(w',v')$}
		\label{graph:example_join_2}
	\end{minipage}%
\end{figure}
\begin{figure}
	\centering
	\begin{tikzpicture}[node distance=2cm]
		\node (a1) {$\ta$};
		\node[right of=a1] (b1) {$\tb$};
		\node[below of=a1] (c1) {$\tc$};
		\node[right of=c1] (d1) {$\td$};
		\node[right of=b1] (a2) {$\te$};
		\node[right of=a2] (b2) {$\tf$};
		\node[below of=a2] (c2) {$\tg$};
		\node[below of=b2] (d2) {$\mathtt{h}$};
		\draw[dotted, line width=1.0pt] (a1) -- (d1);
		\draw[dotted, line width=1.0pt] (a1) -- (c1);
		\draw[dotted, line width=1.0pt] (b1) -- (c1);
		\draw[dashed, line width=1.0pt] (a2) -- (b2);
		\draw[dashed, line width=1.0pt] (b2) -- (c2);
		\draw[dashed, line width=1.0pt] (b2) -- (d2);
		\draw [bend left] (a1) to (a2);
		\draw [bend left] (a1) to (b2);
		\draw (a1) -- (c2);
		\draw (a1) -- (d2);
		\draw (b1) -- (a2);
		\draw [bend left] (b1) to (b2);
		\draw (b1) -- (c2);
		\draw (b1) -- (d2);
		\draw (c1) -- (a2);
		\draw (c1) -- (b2);
		\draw [bend right] (c1) to (c2);
		\draw [bend right] (c1) to (d2);
		\draw (d1) -- (a2);
		\draw (d1) -- (b2);
		\draw (d1) -- (c2);
		\draw [bend right] (d1) to (d2);
	\end{tikzpicture}
	\caption{Join $G(w,v)\nabla G(w',v')$}
	\label{graph:example_join_joined}
	\centering
	\begin{tikzpicture}[node distance=2cm]
		\node (a1) {$\ta$};
		\node[right of=a1] (b1) {$\tb$};
		\node[below of=a1] (c1) {$\tc$};
		\node[right of=c1] (d1) {$\td$};
		\node[right of=b1] (a2) {$\te$};
		\node[right of=a2] (b2) {$\tf$};
		\node[below of=a2] (c2) {$\tg$};
		\node[below of=b2] (d2) {$\mathtt{h}$};
		\draw[dotted, line width=1.0pt] (a1) -- (d1);
		\draw[dotted, line width=1.0pt] (a1) -- (c1);
		\draw[dotted, line width=1.0pt] (b1) -- (c1);
		\draw[dashed, line width=1.0pt] (a2) -- (b2);
		\draw[dashed, line width=1.0pt] (b2) -- (c2);
		\draw[dashed, line width=1.0pt] (b2) -- (d2);
	\end{tikzpicture}
	\caption{Union of $G(w,v)$ and $G(w',v')$}
	\label{graph:example_union}
\end{figure}
	
\begin{remark}	
Note that the well-known graph operation \emph{insertion of a universal vertex} can be implemented in the following way. Since the graph of an isolated vertex $G(\tx,\tx)$ for some $\tx \in \Sigma$ can be joined with a given $2$-word-$\pi$-representable graph: let $w,v \in \Sigma^*$ with $|w|_\ta = |v|_\ta$ for all $\ta \in \al(w)$ and let $\tx \in \Sigma \setminus \al(w)$. Then the insertion of the universal vertex $\tx$ into $G(w,v)$ is given by the join $G(w,v) \nabla G(\tx,\tx) = G(w\tx,v\tx)$ (cf. \Cref{graph:example_insert_universal_vertex}).
\end{remark}

	\begin{figure}
		\begin{minipage}[b][][b]{0.48\textwidth}
		\centering
		\begin{tikzpicture}
			\node (a) {$\ta$};
			\node[right of=a] (b) {$\tb$};
			\node[below of=a] (c) {$\tc$};
			\node[right of=c] (d) {$\td$};
			\node[right of=b] (e) {$\tx$};
			\draw[dotted, line width=1.0pt] (a) -- (d);
			\draw[dotted, line width=1.0pt] (a) -- (c);
			\draw[dotted, line width=1.0pt] (b) -- (c);
			\draw [bend left] (a) to (e);
			\draw (b) -- (e);
			\draw (c) -- (e);
			\draw (d) -- (e);
		\end{tikzpicture}
		\caption{Insertion of universal vertex $\tx$ in $G(w,v)$ from \Cref{graph:example_join_1}}
		\label{graph:example_insert_universal_vertex}
		\end{minipage}%
		\hfill
		\begin{minipage}[b][][b]{0.48\textwidth}
		\centering
		\begin{tikzpicture}
			\node (a) {$\ta$};
			\node[right of=a] (b) {$\tb$};
			\node[below of=a] (c) {$\tc$};
			\node[right of=c] (d) {$\td$};
			\node[right of=b] (e) {$\te$};
			\draw[dotted, line width=1.0pt] (a) -- (d);
			\draw[dotted, line width=1.0pt] (a) -- (c);
			\draw[dotted, line width=1.0pt] (b) -- (c);
		\end{tikzpicture}
		\caption{Insertion of isolated vertex $\te$ in $G(w,v)$ from \Cref{graph:example_join_1}}
		\label{graph:example_insert_isolated_vertex}
		\end{minipage}%
	\end{figure}

\ifpaper
\else
\begin{definition}\label{def:union}
	Let $w,w',v,v'\in\Sigma^\ast$ with $\alph(w)\cap\alph(w')=\emptyset$ s.t. $G(w,v)$, $G(w',v')$ are $2$-word-$\pi$-representable graphs.
	Then $G(w,v)\cup G(w',v')$ is the \textit{union} of the two graphs defined by the set of vertices $V=V(G(w,v))\cup V(G(w',v'))$ and the set of edges $E(G(w,v)\cup G(w',v'))=E(G(w,v))\cup E(G(w',v'))$.
\end{definition}
\fi

Next, we tackle the graph union. Note that the only difference (despite the dropped condition on the same number of the word's letters) is that we use $v'v$ instead of $vv'$.

\ifpaper
\begin{theorem}\label{thm:graph-union}
	We have $G(w,v) \cup G(w',v') = G(ww',v'v)$.
\end{theorem}

\else
\begin{theorem}\label{thm:graph-union}
	We have $G(w,v) \cup G(w',v') = G(ww',v'v)$.
\end{theorem}

\begin{proof}
	By definition, we have $V(G(w,v) \cup G(w',v')) = V(G(ww',v'v))$.
	Thus, we continue with showing that $E(G(w,v) \cup G(w',v')) = E(G(ww',v'v))$.
	Since $\alph(w)\cap\alph(w')=\emptyset$, we have $E(G(w,v))\subseteq E(G(ww',v'v))$. The same applies to every edge from $E(G(w',v'))$ being contained in $E(G(ww',v'v))$. Furthermore, $\pi_{\ta,\tb}(ww')=\ta^{|w|_\ta}\tc^{|w'|_\tc}\neq\tc^{|v'|_\tc}\ta^{|v|_\ta}=\pi_{\ta,\tb}(v'v_1)$ for all $\ta\in\alph(w), \tc\in\alph(w')$. So there are no edges between any $\ta\in\alph(w)$ and $\tb\in\alph(w')$. Thus, $G(ww',v'v')=G(w,v)\cup G(w',v')$ holds.\qed
\end{proof}
\fi

The union of the graphs depicted in \Cref{graph:example_join_1,graph:example_join_2} is $2$-word-$\pi$-represented by
$ww' = \ta\td\tb\tc\te\tg\mathtt{h}\tf$ and $v'v = \mathtt{h}\tg\te\tf\tb\ta\tc\td$ (cf. \Cref{graph:example_union}).





\begin{remark}
Note that the operation \emph{insertion of an isolated vertex} in a graph $G(w,v)$ for suitable $w,v \in \Sigma^*$ is a union of $G(w,v)$ and $G(\tx,\tx)$ for $\tx \notin \al(w)$. Thus, if we insert the isolated vertex $\te$
to the graph (\Cref{graph:example_join_1}), we obtain the graph 
represented by the words $w=\ta\td\tb\tc\te$ and $v=\te\tb\ta\tc\td$ (\Cref{graph:example_insert_isolated_vertex}).
\end{remark}

We finish this work by returning to graphs that are $2$-word-$\pi$-represented by $1$-uniform words. Cographs are a subclass of the permutations graphs (and thus they are $2$-word-$\pi$-representable by $1$-uniform words) and defined as follows \cite{doi:10.1137/1.9780898719796,CorLerBur81}.

\begin{definition}
	A \emph{cograph} is defined inductively as follows:\\
	(1) A graph on a single vertex is a cograph.\\
	(2) If $G_1, G_2, \ldots G_k$ are cographs, then so is their union $G_1 \cup G_2 \cup \cdots \cup G_k$.\\
	(3) If $G$ is a cograph, then so is its complement. 
\end{definition}

A construction of two representing words that relies on the afore mentioned graph operations can be done for cographs.

\begin{construction}\label{construction}
	Let $G$ be a cograph with one vertex, that is $G = (\{\ta\}, \emptyset)$ for a letter $\ta$. Then $w= \ta$ and $v= \ta$
	$2$-word-$\pi$-represent $G$. 
	Consider cographs $G_1$ and $G_2$ with $V(G_1) \cap V(G_2) = \emptyset$. 
	For each $i \in \{1,2\}$, words $w_i,v_i \in V(G_i)^\ast$ exist such that $w_i$ and $v_i$ $2$-word-$\pi$-represent $G_i$. If $G$ is the graph union of $G_1$ and $G_2$, we have that $w = w_1 w_2$ and $v = v_2 v_1$ $2$-word-$\pi$-represent $G$ (\Cref{thm:graph-union}). 
	Consider a cograph $G'$ and words $w$ and $v$ that $2$-word-$\pi$-represent $G'$. If $\overline G$ is the complement graph of $G'$, $w$ and $v^R$ $2$-word-$\pi$-represent $\overline G$ (\Cref{rem:permutation-graphs-complement}). 
	Note that all words constructed this way are 1-uniform. 
\end{construction}

Since the definition of cographs is defined by graph operations, the word-representability of these graphs follows elegantly from the insight of the operaitons on graphs.


	\section{Conclusion}
	In this work, we generalised the notion of word-representability from Kitaev and Pyatkin \cite{KitPya2008} to two words.
We defined that there is an edge between two nodes $\ta$ and $\tb$ iff the projections onto these two letters of both 
words are equal. The first result is that we can represent all graphs with two words both of length at most 
$|V|\cdot|\overline{E}|$ for the nodes of the represented graph $V$ and the edge set $\overline{E}$ 
of the complement graph. This is indeed worse than the well-known adjacency matrix to store graphs. 
Thus, we investigated which graphs are $2$-word-$\pi$-representable by $1$-uniform words and identified 
them as exactly the permutation graphs. This lead to the definition of $\mathcal{G}_k$, the set of 
all graphs $2$-word-$\pi$-representable by two $k$-uniform words. We proceeded by giving some insights 
about these classes and especially $\mathcal{G}_2$. We finished our work by investigating the classic 
graph operations union and join. This leads, interestingly to an elegant way to construct two words 
$2$-word-$\pi$-representing the class of cographs - a subclass of the permutation graphs. It remains 
future work to characterise the classes $\mathcal{G}_k$ in more detail. 
Moreover, the problem of given a graph and a word, determine a second word, such that both together $2$-word-$\pi$-represent the graph is in our opinion an interesting problem that should be pursued. We did not find an example of a graph $G$ such that $G \notin \mathcal{G}_2$. By Remark~\ref{rem:factorial-speed} we know that such a graph exists. Finding such a graph remains an open problem. Furthermore, if the classes $\mathcal{G}_k$ are an infinite hierarchy, another problem that we suggest is to determine the smallest $k$ for a given graph $G$ such that $G \in \mathcal{G}_k$. 
	
	
	\bibliographystyle{splncs04}
	\bibliography{hen}
	





\end{document}